\documentclass[12pt,reqno,oneside]{amsart}
\usepackage{geometry}
\usepackage{xr}

\usepackage[T1]{fontenc}
\usepackage{lmodern}
\usepackage{xcolor}
\usepackage[colorlinks=true, pdfstartview=FitV, linkcolor=black, citecolor=black,
urlcolor=black]{hyperref}
\usepackage{enumerate}
\usepackage{amsmath,amsthm}
\usepackage{graphicx}
\usepackage{longtable}
\usepackage{amssymb}
\usepackage[all,cmtip]{xy}
\usepackage{xypic}
\usepackage{mathrsfs}

\newtheorem{thm}{Theorem}[section]

\newtheorem{lem}[thm]{Lemma}
\newtheorem{prop}[thm]{Proposition}

\theoremstyle{definition}

\theoremstyle{definition}
\newtheorem{rmk}[thm]{Remark}
\newcommand*\diff{\mathop{}\!\mathrm{d}}

\makeatletter

\newcommand{\Rmnum}[1]{\expandafter\@slowromancap\romannumeral #1@}
\makeatother

\DeclareMathOperator{\ord}{ord}

\usepackage{tikz-cd}
\usepackage{tabu}
\newcommand{\Addresses}{{
  \bigskip
  \footnotesize

\textsc{Korteweg-de Vries Instituut, Universiteit van Amsterdam
} \par\nopagebreak
\textsc{Postbus 94248, 1090 GE
Amsterdam, The Netherlands}\par\nopagebreak
  \textit{E-mail address}: \texttt{Z.Zhou@uva.nl}

}}
\usepackage{tikz-cd}
\begin{document}
\title[Existence of curves with prescribed $a$-number]{On the existence of curves with prescribed $a$-number}
\author{Zijian Zhou}

\begin{abstract}
{We study the existence of Artin-Schreier curves with large $ a$-number. We also give  bounds on the $a$-number of  trigonal curves of genus $5$ in small characteristic.}
\end{abstract}
\maketitle

\thispagestyle{empty}
\vspace{-1cm}
\section{Introduction}
Let $k$ be an algebraically closed field of characteristic $p>0$. By a curve we mean a smooth irreducible projective curve defined over $k$. 
Let $X$ be a  curve defined over $k$ and ${\rm Jac}(X)$ be its Jacobian. Such curve has several invariants, e.g. the $a$-number and the $p$-rank.
The $a$-number of the curve $X$ is defined as $a(X)=\dim_k({\rm Hom}(\alpha_p,{\rm Jac}(X)))$ with $\alpha_p$ the group scheme which is the  kernel of Frobenius on additive group scheme $\mathbb{G}_a$.  The $a$-number of $X$ is equal to $g-r$ where $g$ is the genus of $X$ and  $r$ is the rank of the Cartier-Manin matrix, that is, the matrix for the Cartier operator defined on $H^0(X,\Omega_X^1)$. We refer to \cite{MR0084497,Seshadri1958-1959} for the properties of the Cartier operator. The $p$-rank of a curve $X$ is the number $f_X$ such that $\#{\rm Jac}(X)[p](k)=p^{f_X}$. One see that $a(X)+f_X\leq g$. Moreover, a curve is called supersingular if its Jacobian is isogenous to a product of supersingular elliptic curves.

A curve $X$ of genus $g$ is called superspecial if $a(X)=g$.  Ekedahl~\cite{Ekedahl1987} showed that for a superspecial curve $X$ one has $g\leq p(p-1)/2$. We are interested in the existence of  curves in characteristic $p>0$ with  $a$-number close to $g$, namely $a(X)=g-1$ or $g-2$.  For hyperelliptic curves with $p = 2$, Elkin and Pries  \cite{MR3095219} gave a complete description  of  their $a$-numbers. For an Artin-Schreier curve $X$, that is, a $\mathbb{Z}/p\mathbb{Z}$-Galois cover of $\mathbb{P}^1$, of genus $g$ with $p=2$ and $a(X)=g-1$, we prove that $g\leq 3$ and give an explicit  form of the curve. For $p\geq 3$, we show  that  an Artin-Schreier curve with $a$-number $g-1$ has genus $g\leq p(p-1)/2$ and can be written as $y^p-y=f(x)$ with $f(x)$ a polynomial whose degree divides $p+1$, see Proposition \ref{prop 2.2 a=g-1}. Moreover, we have the following.
\begin{thm}\label{AS THEOREM g-1}
Let $k$ be an algebraically closed field with ${\rm char}(k)=p\geq 3$. Let X be an Artin-Schreier curve of genus $g>0$ with equation $y^p-y=f(x)$. If $a(X)=g-1$, then $f(x)\in k[x]$ and if $d=\deg f(x)$ then 
either $p=5,d=3$ and $X$ is isomorphic to a supersingular curve of genus $4$ with equation
\begin{align}
y^5-y=x^3+a_1x,~a_1 \neq 0\, ,\label{Chapter AS p=5 a=g-1}
\end{align}
 or $p=3,d=4$ and $X$ is isomorphic to a supersingular curve of genus $3$ with equation
\begin{align}
y^3-y=x^4+a_2x^2,a_2 \neq 0\, . \label{Chapter AS p=3 a=g-1}
\end{align}
\end{thm}

We prove these results mainly by explicitly calculating the action of the Cartier operator on a basis of holomorphic differential forms. To show the supersingularity we use the de Rham cohomology. 

 By the Deuring-Shafarevich formula \cite{Subrao1975}, an Artin-Schreier curve $X$  has $p$-rank $(m-1)(p-1)$, where  $m$  is the number of branch points. Let $X$ be an Artin-Schreier curve with $a(X)=g-2$. Then for $p\geq 5$, the curve $X$ can be written as  $y^p-y=f(x)$ with $f(x)$ a polynomial. For $p=3$, we give an explicit form of $X$, see Proposition \ref{prop  2.2 a=g-2}. 
Moreover, we have the following. 

\begin{prop}\label{prop a=g-2}
Let $X$ be an Artin-Schreier curve of genus $g>0$ given by an equation $y^p-y=f(x)$, where $f(x)\in k[x]$ and $\deg f(x)=d$. If $d|p+1$ and $a(X)=g-2$, then $p=7,d=4$ and $X$ is isomorphic to the supersingular curve of genus $9$ with equation
$$
y^7-y=x^4+a_1x,\, a_1\in k^*\, .
$$
\end{prop}

 Recall that a result of Re \cite{rre} states that if  $X$ is a non-hyperelliptic curve of genus~$g$, then
\begin{align*}
a(X)\leq \frac{p-1}{p+1}(\frac{2g}{p}+g+1)\, .
\end{align*}
The following results improve Re's bound for  trigonal curves of genus $5$ in low characteristics.  Note  that a trigonal curve of genus $5$ is not hyperelliptic, see, for example, \cite[Section 2.1]{2018arXiv180411277K}.
\begin{thm}\label{thm p=2}
Let $k$ be an algebraically closed field of characteristic $2$. If  $X$ is a trigonal curve of genus 5 defined over $k$, then $a(X)\leq 2$. 
\end{thm} 

\begin{thm}\label{thm p=3}
Let $k$ be an algebraically closed field of characteristic $3$. If $X$ is a trigonal curve of genus 5 defined over $k$, then $a(X)\leq 3$. 
\end{thm}
For $g=5$ and $p=2$, Re's bound says that  $a(X)\leq 3$, while our result implies that $a(X)\leq 2$. Also for $g=5$ and $p=3$, Re's bound says that $a(X)\leq 4$, while our result implies $a(X)\leq 3$.
\section{On the existence of Artin-Schreier curves with prescribed $a$-number}

 Let ${\rm char}(k)=p\geq 3$. Before giving the proof of Theorem \ref{AS THEOREM g-1}, we recall  and prove several results needed for Theorem \ref{AS THEOREM g-1} and give a basis of de Rham cohomology for Artin-Schreier curves.
 
 Since $a(X)+f_X\leq g$, a superspecial curve  has $p$-rank 0. Moreover for superspecial Artin-Schreier curves we have the following result of Irokawa and Sasaki~\cite{MR1118595}.
\begin{thm}\label{AS CURVE zero}
Let $k$ be an algebraically closed field of ${{\rm char}(k)=p\geq 3}$. Let X be a superspecial Artin-Schreier curve with equation $y^p-y=f(x)$, where $f(x) \in k[x]$ and $\deg f(x)=d\geq 2$ with ${\gcd(p,d)=1}$.  Then $X$ is isomorphic to a curve given by $y^p-y=x^d$ with $d | p+1$.
\end{thm}

For the next step, $a(X)=g-1$, we have the following. 
\begin{prop}\label{prop 2.2 a=g-1}
Let $k$ be an algebraically closed field with ${\rm char}(k)=p>0$. Let X be an Artin-Schreier curve of genus $g\geq 1$. If $a(X)=g-1$, then \\
$(1)$ if $p=2$, then $g\leq 3$ and the curve $X$ can be either written as 
$
y^2+y=f(x)
$,
where $f(x)\in k[x]$ and $\deg f(x)=5$ or $7$, or  as $y^2+y=f_0(x)+1/x$ with $\deg f_0(x)=1$ or $3$ and $f_0(x)\in xk[x]$;\\
$(2)$ if $p\geq 3$, then $g\leq (p-1)p/2$ and $X$ is isomorphic to a curve with equation 
\begin{align*}
y^p-y=x^d+a_{d-2}x^{d-2}+\cdots+a_1x, 
\end{align*} 
where $d|p+1$. Moreover if $d=p+1$, then at least one of $a_i$ with $2\leq i\leq d-2$ is non-zero.  If $d<p+1,d~|~ p+1$, then  at least one of $a_i$ with $1\leq i\leq d-2$ is non-zero.
\end{prop}
\begin{proof}
Suppose that $f(x)$ has poles at $\infty, Q_1,\dots,Q_m$ for some $m\in \mathbb{Z}_{\geq 0}$. Let $x-\xi_i$ be a local parameter at $Q_i$. Write $x_i=1/(x-\xi_i)$ for $i=1,\dots,m$ and $x_0=x$.  
Then $f(x)$ can be written as 
\begin{align}
f(x)=f_0(x)+\sum_{i=1}^mf_i(1/(x-\xi_i))=\sum_{i=0}^mf_i(x_i)\, , \label{f(x) AS general p}
\end{align}
where $\deg f_i(x)=d_i$.
By \cite[Lemma 1]{MR0393035}, a basis of $H^0(X,\Omega _X^1)$ is given by $B=\cup_{s=0}^mB_s$ where
\begin{align*}
B_0&=\{x^iy^j\diff x|i,j\in\mathbb{Z}_{\geq 0},~ip+jd\leq(p-1)(d_0-1)-2\} ,\\
B_s&=\{x_s^iy^j\diff x|~i\in\mathbb{Z}_{\geq 1},j\in\mathbb{Z}_{\geq 0},~ip+jd\leq(p-1)(d_s+1)\}, s=1,\dots, m\, .
\end{align*}
The condition $a(X)=g-1$ is equivalent to the rank of the Cartier operator  ${\rm rank}(\mathcal{C})$ being equal to $1$. Note that if $f(x)=\sum_{i=0}^mf_i(x_i)$ as in $(\ref{f(x) AS general p})$, we always have $x_s\diff x\in B$ for $1\leq s\leq m$. Note that $\mathcal{C}(x_s\diff x)\neq 0$ and we get ${\rm rank}(\mathcal{C})\geq m$. \\
$(1)$ For $p=2$, we consider distinct cases for $f(x)$. If $f(x)\in k[x]$ with $\deg f(x)=d$ and $d$ odd, then a basis of $H^0(X,\Omega_X^1)$ is 
\begin{align*}
B&=\{x^iy^j\diff x |~i,j\in\mathbb{Z}_{\geq 0},2i+jd\leq d-3\}=\{x^i\diff x|~i\in\mathbb{Z}_{\geq 0},2i\leq d-3, ~d\geq 3\}.
\end{align*}
Since ${\rm rank}(\mathcal{C})=1$, we must have $x\diff x\in B$ and hence $d\geq 5$. 

Suppose $d\geq 9$, we have $x\diff x$, $x^3\diff x \in B$ and hence
$
\mathcal{C}(x^i\diff x)= x^{(i-1)/2}\diff x
$ for $i=1,3$, a contradiction since ${\rm rank}(\mathcal{C})=1$.
Now if $d=5$ (resp. $7$), then  $B=\{\diff x, x\diff x\}$ (resp. $B=\{\diff x, x\diff x,x^2\diff x\}$) with $\mathcal{C}(\diff x)=\mathcal{C}(x^2\diff x)=0$ and $\mathcal{C}(x\diff x )=\diff x\neq 0$. Hence we have ${\rm rank}(\mathcal{C})=1$ for both cases.

Now if $f(x)=f_0(x_0)+f_1(x_1)$ with $m=1$ as in $(\ref{f(x) AS general p})$, we have $\deg f_0(x_0)=d_0\leq 3$ and $\deg f_1(x_1)=d_1=1$. By putting a pole at $0$ and by scaling, we arrive at $f_1(x_1)=x_1$. Then the curve $X$ can be written as $y^2+y=f_0(x)+1/x$ with $\deg f_0(x)\leq 3$. Again by  a change of coordinates $y \mapsto y+b_0$ with $b_0\in k$ and $b_0^2+b_0=f_0(0)$ we get $f_0(x)\in xk[x]$.\\
$(2)$ If $p\geq 3$, we show the following:\\
$(a)$  For all $p\geq 3$, we have  $d \leq p+1$; \\
$(b)$ If $p=5$ and $d=4$, then ${\rm rank}(\mathcal{C})\geq 2$;\\
$(c)$ If $p\geq 7$ and  $d\geq 3$ with $d\nmid p+1$, then ${\rm rank}(\mathcal{C})\geq 2$.\\
After excluding the cases where ${\rm rank}(\mathcal{C})\geq 2$ or ${\rm rank}(\mathcal{C})=0$, what is left are curves with $a$-number $g-1$.
Note that if $d=2$ and $f(x)\in k[x]$, the curve with equation $y^p-y=f(x)$  is superspecial.

By a change of coordinates, we may assume 
\begin{align*}
f(x)=x^d+a_{d-2}x^{d-2}+\cdots+a_1x+a_0,~ d\geq 3\, .
\end{align*}
$(a)$ If $d\geq p+2$, then by definition we have $x^{p-1}\diff x \in B$.
There exists $l,b\in \mathbb{Z}_{\geq 0}$ such that 
$
d=lp+b$ with $l=1$ and $2\leq b\leq p-1$ or $l\geq 2$ and  $1\leq b\leq p-1
$.
One can show  $x^{p-1-b}y\diff x \in B$ by checking $(p-1-b)p+d\leq (p-1)(d-1)-2$. Then
\begin{align*}
\mathcal{C}(x^{p-1-b}y\diff x)&=\mathcal{C}(x^{p-1-b}(y^p-f(x))\diff x)\\
&=y\mathcal{C}(x^{p-1-b}\diff x)-\mathcal{C}(x^{p-1-b}f(x)\diff x)\neq 0
\end{align*}
as the leading term of $x^{p-1-b}f(x)$ is $x^{lp+p-1}$. This contradiction shows that  $d\leq p+1$.\\
$(b)$ For $p=5$, by $(a)$ we have $d\leq p+1$. Suppose $d\nmid p+1$, we have $d=4$ and 
$y^1\diff x,y^2\diff x \in B$.
Additionally, we have
$
\mathcal{C}(y^i\diff x)=y^{i-1}\diff x
$ for $i=1,2$ and hence 
${\rm rank}(\mathcal{C})\geq 2$, a contradiction. We therefore have $d|p+1$.\\
$(c)$ For $p\geq 7$ and $d\leq p+1$, assume we have $d\nmid p+1$. 	Then there exists $l\in \mathbb{Z}_{>0}$ such that 
$
ld\leq p\leq (l+1)d
$.
Furthermore, we have
$
ld\leq p-1$ and $(l+1)d\geq p+2$ 
as $\gcd(d,p)=1$ and $d\nmid p+1$.
Then there exists $b'$ satisfying 
$
ld+b'=p-1$ for $0\leq b'\leq d-3$.

 If $d=p-1$, then $l=1$, $b'=0$, we get 
$
y\diff x,~y^2\diff x\in B$ and 
$\mathcal{C}(y^i\diff x)=y^{i-1}\diff x
$ for $i=1,2$. This implies ${\rm rank}(\mathcal{C})\geq 2$, a contradiction.
If $d=p-2$, $l=1$ and $b'=1$, then we have
$
i(p-1)\leq (p-1)(p-2)-2$,
which implies $xy\diff x,~xy^2\diff x\in B$. Then $
\mathcal{C}(xy\diff x)
$ and $
\mathcal{C}(xy^2\diff x)
$ are linearly independent and hence ${\rm rank}(\mathcal{C})\geq 2$. Now if $d\leq p-3$, we show that
$
x^{b'}y^{l}\diff x \in B
$.
This is equivalent to showing 
$ld+b'p\leq (p-1)(d-1)-2$. 
By substituting $b'$ with $b'=p-1-ld$ in the inequality, we only need to show
$
d(l+1)(p-1)\geq p^2+1
$,
which is clear since $(l+1)d\geq p+2$.

Now we show that 
$x^{b'}y^{l}\diff x,
x^{b'}y^{l+1}\diff x \in B$. It suffices to show
$ld+d+b'p\leq (p-1)(d-1)-2$. We have $ld+d+b'p\leq (p-1)-b'+d+b'p$ as $b'=p-1-ld$. Hence we only need to show
\begin{align}
d\leq (d-b'-2)(p-1)-2 \label{d<(d-b'-2)(p-1)-2}\, .
\end{align} 
Note that $b'\leq d-3$, we have $(d-b'-2)(p-1)-2\geq p-3\geq d$. Then $x^{b'}y^{l}\diff x,
x^{b'}y^{l+1}\diff x \in B$ and 
\begin{align*}
\mathcal{C}(x^{b'}y^j\diff x)=\sum_{t=0}^{j}(-1)^t{\binom{j}{t}(y^{l-t})\mathcal{C}(x^{b'}f^t(x)\diff x)}=0,~j=l,l+1\, .
\end{align*}
Put $t=l$, then 
$
\mathcal{C}(x^{b'}f^l(x))=\mathcal{C}(x^{b'+ld}+\cdots)\diff x)\neq 0
$,
which implies ${\rm rank}(\mathcal{C})\geq 2$. Therefore we have $d|p+1$.
\end{proof}

Now we will  use the de Rham cohomology $H^1_{dR}(X)$ for a curve $X$ of genus $g$. Recall that this is a vector space of dimension $2g$ provided with a non-degenerate pairing,  cf. \cite[Section 12]{Oort1999}.
Let $X$ be an Artin-Schreier curve over $k$ of genus $g$ with equation 
\begin{align}\label{AS equation p}
y^p-y=h(x)\, ,
\end{align}
where $h(x)\in k[x]\backslash k$ is non-zero of degree $d$. 
Let $\pi : X\to \mathbb{P}^1$ be the $\mathbb{Z}\slash p$-cover. Put $U_1=\pi ^{-1}(\mathbb{P}^1-\{0\})$ and $U_2=\pi ^{-1}(\mathbb{P}^1-\{\infty\})$.  For the open affine cover $\mathcal{U}=\{U_1,U_2\}$, we consider the de Rham cohomology $H^1_{dR}(X)$ as in \cite[Section 5]{MR0241435}, i.e.
$$
H^1_{dR}(X)=Z_{dR}^1(\mathcal{U})/B_{dR}^1(\mathcal{U})
$$
with
$
Z_{dR}^1(\mathcal{U})=\{(t,\omega_1,\omega_2)|t\in \mathcal{O}_X(U_1\cap U_2),\omega_i\in \Omega_X^1(U_i),\diff t=\omega_1-\omega_2\}
$
and 
$B_{dR}^1(\mathcal{U})=\{(t_1-t_2,\diff t_1,\diff t_2)|t_i\in \mathcal{O}_X(U_i)\}$. 

Under the action of Verschiebung operator $V$ on $H^1_{dR}(X)$, one has $V(H^1_{dR}(X))=H^0(X,\Omega_X^1)$ and $V$ coincides with the Cartier operator on $H^0(X,\Omega_X^1)$. 

For $1\leq i \leq g$, put $s(x)=xh'(x)$ with $h'(x)$ the formal derivative of $h(x)$ and write $s(x)=s^{\leq i}(x)+s^{> i}(x)$ with $s^{\leq i}(x)$ the sum of monomials of degree $\leq i$.
Then we have the following proposition.
\begin{prop}\label{prop de Rham basis of y^p-y=h(x)}
Let $X$ be an  Artin-Schreier curve over $k$ with equation 
$
y^p-y=h(x)
$, where $h(x)\in k[x]$ and  $\deg h(x)=d$.
Then $H^1_{dR}(X)$ has a basis with respect to $\mathcal{U}=\{U_1,U_2\}$ consisting of the following residue classes with representatives in $Z_{dR}^1(\mathcal{U})$:
\begin{align}
\alpha_{i,j}&=[(0,x^iy^j\diff x,x^iy^j\diff x)], \\
\beta_{i,j}&=[(\frac{y^{p-1-j}}{x^{i+1}},-\frac{\phi_{i,j} (x,y)}{x^{i+2}}{\rm d}x,\frac{(p-1-j)s^{> i+2}(x)y^{p-2-j}}{x^{i+2}}{\rm d}x)],\label{de rham basis AS H^1}
\end{align}
where  $i,j\in\mathbb{Z}_{\geq 0}, pi+jd\leq (p-1)(d-1)-2$ and  $\phi_{i,j}(x,y)=(p-1-j)s^{\leq i+2}(x)y^{p-2-j}+(i+1)y^{p-1-j}$.

\end{prop}
\begin{proof}
Clearly,  $\omega_{i,j}=x^iy^j\diff x$ form a basis of $H^0(X,\Omega_X^1)$ for $i,j\in \mathbb{Z}_{\geq 0}$ with $pi+dj\leq (p-1)(d-1)-2$.
On the other hand, we may identify $\mathcal{O}_X(U_2)$ with the $k$-algebra $k[x,y]$ defined by $(\ref{AS equation p})$. Moreover, $x^iy^j$ with $i\geq 0, 0\leq j\leq p-1$ form a basis of the image of $\mathcal{O}_X(U_2)$ in $\mathcal{O}_X(U_1\cap U_2)$. Additionally, we have $x^iy^j\in\mathcal{O}_X(U_1)$ for  $0\leq j\leq p-1$ and $-pi\geq dj$. Then the residue classes $[x^iy^j]$ form a basis of $H^1(X,\mathcal{O}_X)$ for $i<0,0\leq j\leq p-1$ and $-pi-dj< 0$. By substituting $i=-(i'+1), j=p-1-j'$, the residue classes $[x^{i+1}y^{p-1-j}]$ form a basis with $i\geq 0,0\leq j\leq p-1$  and $pi+jd\leq (d-1)(p-1)-2$.

Now we check the equality that $\diff f_{i,j}=\omega_{i,j,1}-\omega_{i,j,2}$ for residue classes $\beta_{i,j}=[(f_{i,j},\omega_{i,j,1},\omega_{i,j,2})]$. Note that 
\begin{align*}
\diff f_{i,j}&=\diff \frac{y^{p-1-j}}{x^{i+1}}=\frac{(p-1-j)x^{i+1}y^{p-2-j}\diff y}{x^{2i+2}}-\frac{(i+1)x^{i}y^{p-1-j}\diff x}{x^{2i+2}}\\
&=\frac{-(p-1-j)x^{i+1}y^{p-2-j}h'(x)\diff x}{x^{2i+2}}-\frac{(i+1)x^{i}y^{p-1-j}\diff x}{x^{2i+2}}\\
&=-\frac{\psi_{i,j}(x,y)\diff x}{x^{i+2}}-\frac{(p-1-j)y^{p-2-j}s^{>i+2}(x)\diff x}{x^{i+2}}=\omega_{i,j,1}-\omega_{i,j,2}\, ,
\end{align*}
which ends the proof.
\end{proof}
\begin{rmk}
The pairing $\langle~, ~\rangle$ for this basis is as follows:  $\langle\alpha_{i_1,j_1},\beta_{i_2,j_2} \rangle \neq 0$ if $(i_1,j_1)=(i_2,j_2)$ and $\langle\alpha_{i_1,j_1},\beta_{i_2,j_2} \rangle = 0$ otherwise. Indeed, for $(i_1,j_1)=(i_2,j_2)$ we have $\ord_{\infty}(y^{p-1}/x\diff x)=-1$ and hence $\langle\alpha_{i_1,j_1},\beta_{i_2,j_2} \rangle \neq  0$. For  other cases,  the proof is similar to the proof of \cite[Theorem 4.2.1]{soton373877}.
\end{rmk}
\subsection{Proof of Theorem \ref{AS THEOREM g-1}} Now we prove  Theorem \ref{AS THEOREM g-1} by showing ${\rm rank}(\mathcal{C})=1$.
For $d\leq 2$, the situation is trivial and ${\rm rank}(\mathcal{C})=0$ for all $p>0$. Then we may assume that the polynomial $f(x)$ has the form:
\begin{align*}
f(x)=x^d+a_{d-2}x^{d-2}+\cdots+a_1x\, , d>2\, .
\end{align*}
Also a basis of $H^0(X,\Omega _X^1)$ is given by forms below:
\begin{align*}
B=\{x^iy^j\diff x|ip+jd\leq(p-1)(d-1)-2\}\, .
\end{align*}
$(1)$ For $p\geq 7$, if $a_i=0$ for $ i\in \{1,2,\ldots
,d-2,d\}$, then by Theorem \ref{AS CURVE zero} we have  ${\rm rank}(\mathcal{C})=0$. Otherwise, let $i_0$ be the largest integer in $\{1,2,\ldots
,d-2\}$ such that $a_{i_0}\neq 0$. There are non-negative integers $l$, $m$, $b$ satisfying
$
ld=p+1$
and $d-2=mi_0+b$ with $b\leq i_0-1$.

Suppose  $2\leq i_0\leq d-2$, we show that  
$
x^by^{l-1+m}\diff x\in B\, .
$
This is equivalent to showing
\begin{align*}
bp+(l+m-1)d\leq (d-1)(p-1)-2,
\end{align*}
for $m\geq 1$, $i_0\geq 2$. By substituting $b=d-2-mi_0$, one can show this is equivalent to
$
m(pi_0-d)\geq 2\, ,
$
which is trivial as $d|p+1$ and $m(pi_0-d)\geq 2p-d\geq 2$.

Now if $d=p+1$, then we have $l=1$ and $x^by^m\diff x \in B$ as showed above.  If $b=0$, we have $d-2=p-1=mi_0$. By $i_0\geq 2$, we have $m\leq (p-1)/2$. We show that $y^{m+1}\diff x\in B$ if $p\geq 5$. It is sufficient to show that  $(m+1)\leq (p-1)(d-1)-2=p(p-1)-2$. This is true for $p\geq 5$. Then $\mathcal{C}(y^{m+1}\diff x)\neq 0$ and $\mathcal{C}(y^m\diff x)\neq 0$ are linearly independent. This implies ${\rm rank}(\mathcal{C})\geq 2$ for $b=0$. Suppose $b\geq 1$. We show that 
$
x^{b-1}y^{m+1}\in B\, .
$
Note that $d-2=p-1=mi_0+b$. By a similar fashion, we only need to show 
$
m(i_0-1)(p-1)\geq 4\, , 
$
which is true if $p\geq 5$.
Then
\begin{align*}
\omega_{b,m}:=\mathcal{C}(x^by^m\diff x)=\mathcal{C}(x^b(y^p-f(x))^m\diff x)&=\mathcal{C}((-1)^mx^ba_{i_0}^m(x^{i_0})^m\diff x)+\cdots\\
&=\mathcal{C}((-1)^ma_{i_0}^mx^{p-1}\diff x)+\cdots\neq 0\, .
\end{align*}
Similarly, we have
\begin{align*}
\omega_{b-1,m+1}:&=\mathcal{C}(x^{b-1}y^{m+1}\diff x)=\mathcal{C}(x^{b-1}(y^p-f(x))^{m+1}\diff x)\\
&=\mathcal{C}((m+1)(-1)^{m+1}a_{p+1}a_{i_0}^mx^{2p-1}\diff x)+\cdots\neq 0\, .
\end{align*}
Since $\omega_{b,m}$ and $\omega_{b-1,m+1}$ are $k$-linearly independent, we have ${\rm rank}(\mathcal{C})\geq 2$. 

If $d|p+1$ and $d\leq (p+1)/{2}$, we show that 
$
x^by^{l+m}\diff x \in B\, ,
$
which is equivalent to
$
bp+(l+m)d\leq (d-1)(p-1)-2
$. Since $m(pi_0-d)\geq 2p-d$,
we only need to show
$
m(pi_0-d)-d\leq 2
$, which is true for $p\geq 7$. 
Hence $\mathcal{C}(x^by^{l+m}\diff x)\neq 0$ and $\mathcal{C}(x^by^{l+m-1}\diff x)\neq 0$ by the same method above.

Assume $i_0=1$ and  $a_i=0$ for any $i\in {2,3,\ldots,d-2}$, if $d=p+1$, by a simple change of coordinates the curve is superspecial and ${\rm rank}(\mathcal{C})=0$. Otherwise we have $d<p+1$, in this case we have 
$
d-2=m+b
$.
We show that 
$
y^{l+m+b-1}\diff x, y^{l+m+b}\diff x\in B\, ,
$
which is equivalent to showing
\begin{align*}
(l+m+b-1)d&\leq (d-1)(p-1)-2\text{ and }
(l+m+b)d\leq (d-1)(p-1)-2\, ,
\end{align*}
respectively. These can be simplified to 
\begin{align*}
d^2-(p+2)d+2p+2&\leq 0, \quad
d^2-(p+1)d+2p+2\leq 0\, .
\end{align*}
These two inequalities hold for $p\geq 11$. For $p=7$, we have $d|p+1=8$ and hence $d\geq 4$. Then those two inequalities also hold.

Moreover, we have
\begin{align*}
\mathcal{C}(y^{l+m+b-1}\diff x)&=\mathcal{C}((y^p-f(x))^{l+m+b-1}\diff x)\\
&=\mathcal{C}((-1)^{l+m+b-1}(x^d)^{l-1}(a_1x)^{m+b}\diff x)+\cdots \\
&=\mathcal{C}((-1)^{l+m+b-1}a_1^{m+b}x^{p-1}\diff x)+\cdots
\neq 0 
\end{align*}
and
$
\mathcal{C}(y^{l+m+b}\diff x)=\mathcal{C}((-1)^{l+m+b-1}
a_1^{m+b}x^{p-1}y^p\diff x)+\dots\neq 0  
$.
Then  ${\rm rank}(\mathcal{C})\geq 2$.\\
$(2)$ For $p=5$ and $d=p+1=6$, to get ${\rm rank}(\mathcal{C})=1$ we must have $i_0\geq 2$, otherwise $X$ is superspecial by Theorem \ref{AS CURVE zero}. Then
$
x^{b-1}y^{m+1}\diff x, x^{b}y^{m}\diff x\in B
$
for $b\geq 1$ and $y^{m+1}\diff x, y^{m}\diff x\in B$ for $b=0$ (similar to the case $p=7$).
This implies ${\rm rank}(\mathcal{C})\geq 2$.
As for $d=3$, if $a_{d-2}=a_1=0$, then it is superspecial by Theorem \ref{AS CURVE zero}. If $a_1 \neq 0$, then $y^2\diff x \in B$ and ${\rm rank}(\mathcal{C})= 1$. 

For the supersingularity, let $X$ be a curve given by equation $y^5-y=x^3+a_1x$ with $a_1\neq 0$. Then we have $H^0(X,\Omega_X^1)=\langle\diff x,x\diff x,y\diff x,y^2\diff x\rangle$ and $\mathcal{C}(H^0(X,\Omega_X^1))=\langle\, \diff x\,\rangle$. 
Moreover by using Proposition \ref{prop de Rham basis of y^p-y=h(x)}, one can compute that  $X$ has  Ekedahl-Oort type $[4,3,2]$   and the curve $X$ is supersingular by \cite[Step 2, page 1379]{10.2307/23030376}.  For the definition of Ekedahl-Oort type we refer \cite{Ekedahl2009}.
\\
$(3)$ For $p=3$, if $d=2$ the curve is superspecial. If ${d=4}$, then we may assume that $a_2\neq 0$ in $f(x)$, otherwise by a simple change of coordinates we may assume the curve is  given by equation $y^3-y=a_4x^4$, which is superspecial by the Theorem \ref{AS CURVE zero}. 

If $a_2\neq 0$, then by a change of coordinate we get $f(x)=x^4+a_2x^2$. A basis of $H^0(X,\Omega_X^1)$ is $\{\diff x, x\diff x,y\diff x\}$ with $\mathcal{C}(\diff x)=\mathcal{C}(x\diff x)=0$ and
$
\mathcal{C}(y\diff x)=\mathcal{C}(-a_2x^2\diff x)=-a_2^{1/3}\diff x\neq 0\, .
$
This implies ${\rm rank}(\mathcal{C})=1$.
Similarly using Proposition \ref{prop de Rham basis of y^p-y=h(x)}, a curve given by equation $y^3-y=x^4+a_2x^2$ with $a_2\neq 0$ has Ekedahl-Oort type $[3,2]$ and hence is supersingular by \cite[Step 2, page 1379]{10.2307/23030376}. 

\subsection{Proof of Proposition \ref{prop a=g-2}}

Let $X$ be an Artin-Schreier curve  given by equation $y^p-y=f(x)$ with $\deg f(x)=d|p+1$ and ${\rm rank}(\mathcal{C})=2$. We now give the proof of Proposition \ref{prop a=g-2} by showing  ${\rm rank}(\mathcal{C})=2$. 
We may assume that the polynomial $f(x)$ has the form:
\begin{align*}
f(x)=x^d+a_{d-2}x^{d-2}+\cdots+a_1x\, .
\end{align*} 
By the proof of Theorem \ref{AS THEOREM g-1}, there is an integer $n\in\{1,2,\dots,d-2\}$ such that $a_{n}\neq 0$.  Again denote by $i_0$ the largest integer in $\{1,2,\dots,d-2\}$ such that $a_{i_0}\neq 0$ and let $l,m,b$ be the same as in the proof of Theorem \ref{AS THEOREM g-1}.

For $p\geq 7$, if $d=p+1$, we show that ${\rm rank}(\mathcal{C})\geq 3$. Indeed by Theorem \ref{AS CURVE zero}, we have $i_0\geq 2$ and $d-2=p-1=mi_0+b$. If $b=0$, then $d-2=p-1=mi_0$ and $m\leq (p-1)/2$. Moreover from the proof of Theorem \ref{AS THEOREM g-1}, part $(1)$, we have $y^m\diff x, y^{m+1}\diff x\in B$. We show that $y^{m+2}\diff x\in B$. It suffices to show that $(m+2)d\leq (p-1)(d-1)-2$, which is equivalent to showing $(m+2)(p+1)\leq p^2-p-2$ for any $1\leq m\leq (p-1)/2$. This is true for $p\geq 7$. On the other hand, note that $\mathcal{C}(y^m\diff x),\mathcal{C}(y^{m+1}\diff x)$ and $\mathcal{C}(y^{m+2}\diff x)$ are linearly independent. Then ${\rm rank}(\mathcal{C})\geq 3$ in this case. Now if $b\geq 1$, we showed that $x^by^m\diff x,x^{b-1}y^{m+1}\diff x\in B$. By a similar argument as in the case $b=0$ above, one can show that $x^by^{m+1}\diff x \in B$. Additionally, $\mathcal{C}(x^by^m\diff x),\mathcal{C}(x^{b-1}y^{m+1}\diff x)$ and $\mathcal{C}(x^by^{m+1}\diff x)$ are linearly independent. Then we have ${\rm rank}(\mathcal{C})\geq 3$ for $p\geq 7$ and $d=p+1$.

Now if $d|p+1$ and $d<p+1$, then $l=(p+1)/d \geq 2$. If $i_0\geq 2$,  we show that ${\rm rank}(\mathcal{C})\geq 3$ for $p\geq 7$. Note that we have $x^by^{l+m-1}\diff x,x^by^{l+m}\diff x \in B$ by the part (1) of the proof of Theorem \ref{AS THEOREM g-1}. We now claim that $x^by^{l+m+1}\diff x \in B$. By definition of $B$, it suffices to show 
$$
(l+m+1)d+bp\leq (p-1)(d-1)-2\, .
$$
By substituting $b=d-2-mi_0$ and $p=ld-1$, the inequality can be simplified to $(i_0l-1)m\geq 3$. This is true as $i_0\geq 2,l\geq 2$ and $m\geq 1$. For $i_0=1$, we show that ${\rm rank}(\mathcal{C})\geq 3$ for $p\geq 11$. Note that in this case we have $d-2=m$. One can easily show that $y^{l+m}\diff x,y^{l+m-1}\diff x \in B$ by the definition of $B$. Additionally, we show that $y^{l+m+1}\diff x\in B$ for $p\geq 11$. Indeed, it suffices to show $(l+m+1)d\leq (p-1)(d-1)-2$, which can be simplified to $2l+d\leq p$. Note that $ld=p+1$, we only need to show $2(p+1)/d+d\leq p$ which can be rewritten as $d^2-dp+2(p+1)\leq 0$. This is true for $3\leq d \leq (p+1)/2$. 

For $p=7$ and $i_0=1$, we have $d=4$ and the curve is given by equation $y^7-y=x^4+a_1x$ with $a_1\in k^*$. Then 
$$B=\{x^iy^j\diff x,|i,j\in \mathbb{Z}_{\geq 0},7i+4j\leq 16\}$$
and $\mathcal{C}(x^iy^j\diff x)=0$ for all $i,j$ except $(i,j)=(0,4),(0,3),(1,2)$. Moreover, $\mathcal{C}(y^4\diff x)$ and  $\mathcal{C}(y^3\diff x)$ are linearly independent and $\mathcal{C}(y^3\diff x)=\xi \mathcal{C}(xy^2\diff x)$ for some $\xi\in k^*$. Then ${\rm rank}(\mathcal{C})=2$. Using Proposition \ref{prop de Rham basis of y^p-y=h(x)}  and by  \cite[Step 2, page 1379]{10.2307/23030376} as above,  the curve is supersingular.

Now let $p=5$. If $d=3$, then by Theorem \ref{AS CURVE zero} and Theorem \ref{AS THEOREM g-1}, we have $a(X)=g$ or $g-1$. For $d=6$, we get ${d-2=4=mi_0+b}$. Additionally for $i_0=2,3,4$, one can easily show that $y^2\diff x,y^3\diff x,xy^3\diff x \in B$ and $\mathcal{C}(y^2\diff x),\mathcal{C}(y^3\diff x)$ and $\mathcal{C}(xy^2\diff x)$ are linearly independent. Hence ${\rm rank}(\mathcal{C})\geq 3$ and $a(X)\leq g-3$ with $g=10$.

For $p=3$ and $d\, |\, p+1=4$, by Theorem \ref{AS CURVE zero} and Theorem \ref{AS THEOREM g-1}, we have $a(X)\geq g-1$.

\medskip

Similar to  Proposition \ref{prop 2.2 a=g-1}, we have the following.
\begin{prop}\label{prop  2.2 a=g-2}
Let $k$ be an algebraically closed field with ${\rm char}(k)=p\geq 3$. Let X be an Artin-Schreier curve of genus $g\geq 1$ with equation $y^p-y=f(x)$, where $f(x) \in k(x)$. If $a(X)=g-2$, then \\
$(1)$ if $p=3$, then $g\leq 7$ and the curve $X$ can be either written as 
$
y^3-y=f(x)
$,
where $f(x)\in k[x]$ and $\deg f(x)\leq 8$, or  as $y^3-y=f_0(x)+f_1(1/x)$ with $f_0(x),f_1(x)\in k[x]$ and $\deg f_0(x)\leq 4,\deg f_1(x)\leq 2$;\\
$(2)$ if $p\geq 5$, then $g\leq (2p+1)(p-1)/2$ and  $X$ is isomorphic to a curve with equation 
\begin{align*}
y^p-y=f(x),\, f(x)\in k[x]\, .
\end{align*} 
\end{prop}
The proof of part $(1)$ is similar to the part $(1)$ of the proof of Proposition \ref{prop 2.2 a=g-1} and hence we omit it. One can prove part $(2)$  using the the Deuring-Shafarevich formula.

\section{On the existence of trigonal curves with prescribed $a$-number }

Now we study the existence of trigonal  curves with prescribed $a$-number and give proofs of Theorem \ref{thm p=2} and \ref{thm p=3}. We deal here with genus $5$. 
 It is well known that  a trigonal curve $X$ of genus $5$  is a normalization of a quintic curve $C$ in $\mathbb{P}^2$ with an unique singular point   \cite[Exercise I-6, page 279]{MR770932}, see also \cite[Lemma 2.2.1]{2018arXiv180411277K}.

\subsection{Set up}
For a trigonal curve $X$ of genus $5$ defined over $k$, let $\phi: X\to \mathbb{P}^1$ be a morphism of degree $3$. Then using the base point free pencil trick and Clifford Theorem one can easily show that  $\phi$ is unique (up to  isomorphism of $\mathbb{P}^1$) and $X$ is not hyperelliptic. 
\begin{lem}\label{lem 2.2.1}
Let $p$ be  either $2$ or $3$. If $X$ is a trigonal curve of genus $5$ over $k$, then \\
$(1)$ $X$ is a normalization of  a quintic curve $C$ in $\mathbb{P}^2$ with  an unique singular point of multiplicity $2$. Moreover,

$(i)$ If $C$ has a node, then $C$ is given by a homogeneous polynomial $F\in k[x,y,z]$ of degree $5$ with 
\begin{align*}
F=xyz^3+f\, ,
\end{align*}
where $f$ is a sum of monomials  not divisible by $z^3$.

$(ii)$ If $C$ has a cusp, then $C$ is given by a homogeneous polynomial $F\in k[x,y,z]$ of degree $5$ with 
\begin{align*}
F=x^2z^3+f\, ,
\end{align*}
where $f$ is a sum of monomials not divisible by $z^3$ and the coefficient of $y^3z^2$ in $f$ is non-zero.\\
$(2)$ The normalization of any  $C$ with one singular point in $(i)$ and $(ii)$  is a trigonal curve of genus $5$.
\end{lem}
\begin{proof}
Kudo and Harashita proved the lemma for $p\neq 2$ in \cite[Lemma 2.2.1]{2018arXiv180411277K}. For $p=2$, we show that the part $(1)$ is true and since the proof of the other part is similar to the case $p\geq 3$  we omit it. 

Assuming the singular point  is $(0:0:1)$, the curve $C$ is given by 
$
F=Qz^3+f,
$
where $Q$ is a quadratic form in $k[x, y]$ and $f$ is a sum of monomials  in $x,y$ of  degree $>2$.  

If $Q$ is non-degenerate, then $C$ has a node. Writing the form $Q$ as   $xy$,  we arrive at $F=xyz^3+f$ with $f$ a sum of monomials in $x,y$ of degree $>2$. 

If $Q$ is degenerate, then $C$ has a cusp. Writing the form $Q$  as $x^2$, we arrive at  $F=x^2z^3+f$ with $f$ a sum of monomials in $x,y$ of degree $>2$. 
\end{proof}

We recall the following proposition.
\begin{prop}\cite[\text{Proposition~$2.3.1$}]{2018arXiv180411277K}\label{prop 2.3.1}
Let $X$ be a trigonal curve of genus $5$ defined over $k$.
Let $C$ be an associated quintic curve in $\mathbb{P}^2$ given by Lemma \ref{lem 2.2.1}. Let $h_{l,m} \,(1 \leq  l, m \leq  5)$
be the coefficient of the monomial
$x^{pi_{l}-i_{m}}y^{pj_{l}-j_{m}}z^{pk_{l}-k_{m}}$
in $F^{p-1}$, where 
$$
\begin{tabular}{cccccc}
$l$              & $1$ & $2$ & $3$ & $4$ & $5$  \\
\hline
$i_l$             & $3$ & $1$ & $2$ & $2$ & $1$ \\
$j_l$            & $1$ & $3$ & $2$ & $1$ & $2$  \\
$k_l$          & $1$ & $1$ & $1$ & $2$ & $2$  \\
\end{tabular}\, .
$$
Then the Hasse-Witt matrix $H$ of $X$ is given by $H=(h_{l,m})$.
\end{prop}

\subsection{The proof of Theorem $\ref{thm p=2}$}

Let $p=2$ and $X$ be a trigonal curve of genus $5$ defined over $k$. we start by simplifying the defining equation of the singular model $C\subset \mathbb{P}^2$ of $X$. 
\begin{lem}\label{prop 3.3.1}
Let $k$ be an algebraically closed field with ${{\rm char} (k)=2}$. In the notation of Lemma \ref{lem 2.2.1} case $(i)$, we can choose $f$ as
\begin{align}
f=&(x^3+b_1y^3)z^2+\sum_{i=1}^5(a_ix^{5-i}y^{i-1})z+\sum_{i=1}^{11}a_ix^{11-i}y^{i-6}
 \label{p=2 node  f b_0 neq 0}
\end{align}
or
\begin{align}
f=\sum_{i=1}^5(a_ix^{5-i}y^{i-1})z+\sum_{i=1}^{11}a_ix^{11-i}y^{i-6}\, .\label{p=2 node f b_0=b_1=0}
\end{align}
For case $(ii)$, we can choose $f$ as
\begin{align}
f=y^3z^2+\sum_{i=1}^5(a_ix^{5-i}y^{i-1})z
+\sum_{i=1}^{11}a_ix^{11-i}y^{i-6}\, . \label{p=2 cusp}
\end{align}
\end{lem} 
\begin{proof}
For the case $(i)$ of Lemma \ref{lem 2.2.1}, the curve $C$ is given by $F=xyz^3+f$, where $f$ is the sum of monomials,  which have  degree $>2$ in $x,y$. By a linear transformation $z\mapsto z+\alpha x+\beta y$, we may assume the coefficients of $x^2yz^2$ and $xy^2z^2$ are zero.
Then 
\begin{align*}
f=(b_0x^3+b_1y^3)z^2+\sum_{i=1}^5 a_ix^{5-i}y^{i-1}z+\sum_{i=1}^{11}a_ix^{11-i}y^{i-6}\, ,
\end{align*}
where $b_0,b_1,a_1,\dots,a_{11}\in k$.
 Note that if $(b_0,b_1)\neq (0,0)$, by symmetry we may assume $b_0\neq 0$. By scaling $x\mapsto \alpha x, y \mapsto \beta y$ with $\alpha\beta=1$ and $\alpha^3=1$. Then we may assume  $b_0=1$. On the other hand, if $b_0=b_1=0$ in $f$, then  we have
 \begin{align*}
f=\sum_{i=1}^5a_ix^{5-i}y^{i-1}z+\sum_{i=1}^{11}a_ix^{11-i}y^{i-6}\, .
\end{align*}
For the case $(ii)$ of Lemma \ref{lem 2.2.1}, the curve $C$ is given by $F=x^2z^3+f$, where $f$ is the sum of monomials,  which have degree $>2$ in $x,y$ and the coefficient of $y^3z^2$ is non-zero. Consider $y \mapsto y+\gamma x$ and then consider  $z\mapsto z+\alpha x+\beta y$, we may assume the coefficients of $x^3z^2,x^2yz^2$ and $xy^2z^2$ are zero. Moreover, by  scaling $y\mapsto \delta y$  with $\delta^3=1$, we may assume the coefficient of $y^3z^2$ is equal to 1. Then we have
\begin{align*}
f=y^3z^2+\sum_{i=1}^5a_ix^{5-i}y^{i-1}z+\sum_{i=1}^{11}a_ix^{11-i}y^{i-6}\, ,~ a_1,\dots,a_{11}\in k\, .
\end{align*}

\end{proof}
Now we can give a proof of Theorem \ref{thm p=2}.
\begin{proof}[Proof of Theorem $\ref{thm p=2}$]
Let $C$ be a singular model of $X$ given by Lemma \ref{lem 2.2.1}. If  $C$ has a node, then by Lemma \ref{prop 3.3.1}, $f$ is either given by equation $(\ref{p=2 node  f b_0 neq 0})$ or equation~${(\ref{p=2 node f b_0=b_1=0})}$. If $f$ is given by equation $(\ref{p=2 node  f b_0 neq 0})$, then by Proposition \ref{prop 2.3.1}, the Hasse-Witt matrix $H$ of $X$ is equal to
\begin{align}
\left ( \begin{array}{ccccc}
a_2 & 0 & a_1 & a_7 & a_6 \\ 
0 & a_4 & a_5 & a_{11} & a_{10} \\
a_4 & a_2 & a_3 & a_9 &a_8 \\
1 & 0 & 0 &0 & 1\\
0 & 1 & 0 & b_1 & 0
\end{array}\right )\, .\label{Cartier Manin p=2 node b_0 neq 0}
\end{align}
Let $e_i$ be the $i$-th row of $H$. 
Then    $e_4$ and $e_5$ are linearly independent and ${\rm rank}(H)\geq 2$.

Now we show ${\rm rank}(H)\geq 3$ in this case. Indeed, if ${\rm rank}(H)= 2$, then $e_i$  for $i=1,2,3$ is a linear combination of $e_4$ and $e_5$. By the shape of $H$, we have 
\begin{align*}
a_1=a_3=a_5=a_7=a_{10}=0, a_4=a_8,\, a_2=a_6,\,  b_1a_4=a_{11},\, b_1a_2=a_9\, .
\end{align*}
Hence $C$ is given by 
\begin{align*}
F&=xyz^3+(x^3+b_1y^3)z^2+(a_2x^3y+a_4xy^3)+a_2x^5+a_4x^3y^2+b_1a_2x^2y^3+b_1a_4y^5\\
&=(z+a_2^{1/2}x+a_4^{1/2}y)^2(b_1y^3+x^3+xyz)
\end{align*}
and $C$ is reducible.
This contradiction shows that ${\rm rank}(H)\geq 3$.

Now if $f$ is given by equation $(\ref{p=2 node f b_0=b_1=0})$, then again by Proposition \ref{prop 2.3.1} the Hasse-Witt matrix $H$ of $X$ is equal to
\begin{align}
\left ( \begin{array}{ccccc}
a_2 & 0 & a_1 & a_7 & a_6 \\ 
0 & a_4 & a_5 & a_{11} & a_{10} \\
a_4 & a_2 & a_3 & a_9 &a_8 \\
1 & 0 & 0 &0 & 0\\
0 & 1 & 0 & 0 & 0
\end{array}\right )\, .\label{Cartier Manin p=2 node b_0 =b_1=0}
\end{align}
Then we have ${\rm rank}(H)\geq 2$. Moreover, if ${\rm rank}(H)= 2$, then we have 
$a_i=0$ for all $i\in \{1,\dots,11 \}$ with $i\neq2,4$. This implies 
\begin{align*}
F=xyz^3+a_2x^3yz+a_4xy^3z=xy(z^3+a_2x^2z+a_4y^2)\, ,
\end{align*}
a contradiction. Hence we have  ${\rm rank}(H)\geq 3$ if $C$ has a node. 

If the curve $C$ has a cusp, then by Lemma \ref{prop 3.3.1}, $f$ is given by equation $(\ref{p=2 cusp})$. Hence by Proposition \ref{prop 2.3.1}, the Hasse-Witt matrix of $X$ is equal to  
\begin{align}
\left ( \begin{array}{ccccc}
a_2 & 0 & a_1 & a_7 & a_6 \\ 
0 & a_4 & a_5 & a_{11} & a_{10} \\
a_4 & a_2 & a_3 & a_9 &a_8 \\
0 & 0 & 1 &0 & 0\\
0 & 0 & 0 & 1 & 0
\end{array}\right )\, .\label{Cartier Manin p=2 cusp}
\end{align}
Then we still have ${\rm rank}(H)\geq 2$.  We show ${\rm rank}(H)\geq 3$ by showing that $C$ has at least two singular points if ${\rm rank}(H)= 2$. Indeed, suppose ${\rm rank}(H)= 2$. By $(\ref{Cartier Manin p=2 cusp})$, we obtain $a_2=a_4=a_6=a_8=a_{10}=0$. This implies 
\begin{align*}
F=x^2z^3+y^3z^2+(a_1x^4+a_3x^2y^2+a_5y^4)z+a_7x^4y+a_9x^2y^3+a_{11}y^5\, .
\end{align*}
Denote by $F_x$ (resp. $F_y ,\, F_z$) the formal partial derivative with respect to the variable $x$ (resp. $y, \, z$). Note that we have 
\begin{align*}
F_{x}=0, \, F_y=y^2z^2+a_7x^4+a_9x^2y^2+a_{11}y^4,\,
F_{z}=x^2z^2+a_1x^4+a_3x^2y^2+a_5y^4\, .
\end{align*}
By setting $x=1$ in $F_x, F_y$ and $F_z$, one can easily show that  $(1:a:b)$ is a singular point. Then there are at least two singular points on $C$. By the genus formula for plane curves, the genus of $X$ is less than $5$, a contradiction.

Now we have ${\rm rank}(H)\geq 3$ for any trigonal curve $X$ of genus $5$ over $k$. Then $a(X)\leq 2$. 
\end{proof}

\subsection{The proof of Theorem \ref{thm p=3}}
Let $p=3$ and $X$ be a trigonal curve of genus $5$ defined over $k$. We now give the reductions of the defining equations of the singular model $C\subset \mathbb{P}^2$ of $X$ given by Lemma \ref{lem 2.2.1}.
\begin{lem}\label{prop 3.3.1 p=3}
Let $k$ be an algebraically closed field with ${{\rm char} (k)=3}$. In the notation of Lemma \ref{lem 2.2.1} case $(i)$, we can choose $f$ as  
\begin{align}
(b_0x^3+b_1y^3+b_2x^2y+b_3xy^2)z^2+\sum_{i=1}^5 a_ix^{5-i}y^{i-1}z
+a_6x^5+\sum_{i=8}^{11}a_ix^{11-i}y^{i-6}. \label{p=3 node  f}
\end{align}
For the case $(ii)$, we can choose $f$ as 
\begin{align}
(y^3+\sum_{i=2}^3 b_ix^{4-i}y^{i-1})z^2+\sum_{i=1}^5a_ix^{5-i}y^{i-1}z
+a_7x^4y+a_8x^3y^2+a_{10}xy^4+a_{11}y^5.\label{p=3 cusp}
\end{align}
\end{lem} 
\begin{proof}
If $C$ has a node, then by Lemma \ref{lem 2.2.1}, $F=xyz^3+f$ with $f$ the sum of monomials,  which have  degree $>2$ in $x,y$. By a linear transform $z\mapsto z+\alpha x+\beta y$ we may assume the coefficient of $x^4y$ and $xy^4$ is zero. Then $f$ is equal to
\begin{align*}
(b_0x^3+b_1y^3+b_2x^2y+b_3xy^2)z^2+\sum_{i=1}^5a_ix^{5-i}y^{i-1}z  +a_6x^5+a_8x^3y^2+a_9x^2y^3+a_{11}y^5\, . \
\end{align*}
By Lemma \ref{lem 2.2.1}, if $C$ has a cusp,  then $F=xyz^3+f$ with $f$ the sum of monomials,  which have  degree $>2$ in $x,y$ and the coefficient of $y^3z^2$ is non-zero. By a linear transform $z\mapsto z+\alpha x+\beta y$, we may assume the coefficient of $x^5$ and $x^2y^3$ is zero. Moreover, by scaling $y \mapsto \delta y$, we may assume the coefficient of $y^3z^2$ is $1$. Then we have
\begin{align*}
f=(y^3+b_2x^2y+b_3xy^2)z^2+\sum_{i=1}^5a_ix^{5-i}y^{i-1}z+a_7x^4y+a_8x^3y^2+a_{10}xy^4+a_{11}y^5\, .
\end{align*}
\end{proof}

\begin{proof}[Proof of Theorem $\ref{thm p=3}$]
Let $C$ be the singular model given by Lemma \ref{lem 2.2.1}. Denote by $H$ the Hasse-Witt matrix of $X$ given by Proposition \ref{prop 2.3.1} and by $e_i=(e_{i,1},\dots,e_{i,5})$ the $i$-th row of $H$. Then we have ${{\rm rank}(H)\geq 1}$ because of the Ekedahl's genus bound for superspecial curve \cite{Ekedahl1987}. Suppose ${{\rm rank}(H)= 1}$. We consider different cases for the singular point of $C$. 

If the curve $C$ has a node, by Lemma \ref{prop 3.3.1 p=3}, $f$ is given by equation $(\ref{p=3 node  f})$. If at least one of $b_0,b_1$ is non-zero, by symmetry we may assume $b_0\neq 0$. By scaling we may assume $b_0=1$. Moreover, by Proposition \ref{prop 2.3.1}, we have  $e_4=(2b_2,0,2,b_2^2+2b_3+2a_2,b_2+2a_1)$ which is non-zero. Then $e_i=\lambda_i e_4$ with $\lambda_i\in k$ for $i=1,2,3,5$. In particular, we have $$e_5=(0,2b_3,2b_1,2b_1b_3+2a_5,2b_1b_2 + b_3^2 + 2a_4)=\lambda_5 e_4\, .$$ This implies that  $b_3=0$ and $b_1b_2=0$. 

If $b_1=0$, then $e_5=(0,0,0,2a_5,2a_4)$ is the zero vector. Hence $a_4=a_5=0$. Note that in this case we have $e_{3,1}=2a_{11}$ and $e_{3,1}=\lambda_3 e_{4,1}=0$. Then $a_{11}=0$ and  
\begin{align*}
F&=xyz^3+(x^3+b_2x^2y)z^2+(a_1x^4+a_2x^3y+a_3x^2y^2)z+a_6x^5+a_8x^3y^2+a_9x^2y^3\, .
\end{align*}
One can easily check that $(0:0:1)$ and $(0:1:0)$ are common zeros of $F=F_x=F_y=F_z=0$. Then $C$ has at least two singular points, a contradiction.

Now if $b_1\neq 0$, then the equation $b_1b_2=0$ implies $b_2=0$. Consider a change of coordinate $x\mapsto \alpha x , y\mapsto \beta y$ and multiply $F$ by $1/(\alpha \beta)$. The coefficients of $x^3$ and $y^3$ in $F$ become  $\alpha^2/\beta$ and $\beta^2b_1/\alpha$. By taking $\beta=\alpha^2$ and $\alpha=b_1^{-1/3}$, we may assume $b_0=b_1=1$. Then $e_5=\lambda_5 e_4$ implies $a_2=a_5$ and $a_1=a_4$.
Additionally, we have 
\begin{align*}
e_1&=(2a_1a_3+a_2^2+2a_8,a_1^2+2a_6,2a_1a_2,2a_1a_8+2a_3a_6,2a_2a_6),\\
e_2&=(a_2^2+2a_{11},a_1^2+2a_2a_3+2a_9,2a_1a_2,2a_1a_{11},2a_2a_9+2a_3a_{11})\, .
\end{align*}
Then by $e_1=\lambda_1e_4$ and $e_2=\lambda_2e_4$, we have 
\begin{align*}
a_6=a_1^2,a_8=a_2^2-a_1a_3,a_9=a_1^2-a_2a_3,a_{11}=a_2^2\, .
\end{align*} 
If $a_3=0$, then one can easily  show that $(-1:1:0)$ and $(0:0:1)$ are  common zeros of  $F=F_x=F_y=F_z=0$ and hence are singular points of $C$, a contradiction.
If $a_3\neq 0$, then  $F_z=2(x^3+y^3)z+a_1x^4+a_2x^3y+a_3x^2y^2+a_1xy^3+a_2y^4$ and by substituting

$$
z=(a_1x^4+a_2x^3y+a_3x^2y^2+a_1xy^3+a_2y^4)/(x+y)^3
$$
 in $F_x$ and $F_y$ and by letting $y=1$, we have $(x(x+y)^9F_x)|_{y=1}=((x+y)^9F_y)|_{y=1}$ and  
\begin{align*}
((x+y)^9F_x)|_{y=1}&=(a_1^3 + a_1a_2)x^{12} + (a_1a_2 + a_2^3 + 2a_3^2)x^9 + (a_3^3 + a_3^2)x^6 \\
&+ (a_1^3 +
    a_1a_2 + 2a_3^2)x^3 + a_1a_2 + a_2^3\, .
\end{align*}
Since $a_3\neq 0$, one can check that it has solutions with $x\neq -1,0$.
  Then there exists another singular point on $C$ which is distinct from $(0:0:1)$, a contradiction.

Now if $b_0=b_1=0$, we have $e_4=(2b_2,0,0,b_2^2+2a_2,2a_1)$ and $e_5=(0,2b_3,0,2a_5,b_3^2+2a_4)$.  Since ${\rm  rank}(H)=1$, we have at least one of $b_2,b_3$ is zero. By symmetry we may assume $b_3=0$. If $b_2\neq 0$, by scaling we may assume $b_2=1$. Note that we have  $e_i=\lambda_i e_4$ with $\lambda_i\in k$ for $i=1,2,3,5$. In particular for $i=5$, it is straightforward to see that $\lambda_5=0$ and $a_4=a_5=0$. Moreover, $e_{2,2}=2a_{11}=\lambda_2e_{4,2}=0$. Then by a similar fashion, one can show that $(0:0:1)$ and $(0:1:0)$ are singular points of $C$, a contradiction. If $b_3=b_2=0$, then 
\begin{align*}
e_1&=(2a_1a_3+a_2^2,a_1^2,2a_1a_2,2a_1a_8+2a_3a_6,2a_2a_6), \,
e_4=(0,0,0,2a_2,2a_1)\, .
\end{align*}
This implies $a_1=a_2=0$ otherwise $e_1$ and $e_4$ are linearly independent. Similarly one has $a_4=a_5=0$ by checking the linearly independence of $e_2$ and $e_5$. Now we have $e_4=e_5=0$ and 
 \begin{align*}
e_1&=(0,0,0,2a_3a_6,0),\,
e_2=(0,0,0,0,2a_3a_{11}), \, 
e_3=(0,0,a_3,2a_3a_9,2a_3a_8) \, .
\end{align*}
Since ${\rm rank}(H)=1$, we obtain that $a_6=a_{11}=0$ and $F=y(xz^3+a_3x^2yz+a_8x^3y+a_9x^2y^2)$, a contradiction.

Now if $C$ has a cusp, then by Lemma \ref{prop 3.3.1 p=3}, $f$ is given by equation $(\ref{p=3 cusp})$. Moreover, by Proposition \ref{prop 2.3.1}, we have 
\begin{align*}
e_4=(2b_3,0,2b_2,b_2^2+2a_3,2a_2),~
e_5=(0,2,0,2b_3,2b_2+b_3^2+2a_5)\, .
\end{align*}
If ${\rm rank}(H)=1$, then $e_i=\lambda_ie_5$ with $\lambda_i\in k$ for $i=1,2,3,4$. In particular,  $e_4=\lambda_4e_5$ implies
$
\lambda_4=b_2=b_3=a_2=a_3=0
$.
Then we obtain 
\begin{align*}
e_1&=(0,a_1^2,0,2a_1a_8,2a_1a_7),\,
e_2=(a_5^2+2a_{11},a_4^2,2a_4a_5+2a_{10},2a_4a_{11}+2a_5a_{10},2a_4a_{10}),\\
e_3&=(2a_{8},2a_1a_4,2a_1a_5+2a_7,2a_1a_{11}+2a_4a_8+2a_5a_7,2a_1a_{10}+2a_4a_7) \, .
\end{align*}
Hence by $e_i=\lambda_i e_5$ for $i=1,2,3$, we get 
$
a_7=2a_1a_5,a_8=0,a_{10}=2a_4a_5,a_{11}=a_5^2
$.
Additionally, one can easily check that  $F=F_x=F_y=F_z=0$
have common zeros $(0:0:1), (0:1:0)$ if $a_5=0$ and $(0:0:1), (0:1/(2a_5):1)$ if $a_5\neq 0$, a contradiction. Then ${\rm rank}(H)\geq 2$ if $C$ has a cusp.

In any case, we have ${\rm rank}(H)\geq 2$ and hence $a(X)\leq 3$.
\end{proof}

\bibliographystyle{plain}

\Addresses
\end{document}